\documentclass[final]{siamart0516}
\newsiamthm{remark}{Remark}
\newsiamthm{example}{Example}
\usepackage{amssymb}
\usepackage{mathabx}
\newcommand{\T}{\mathcal{T}}
\newcommand{\TT}{\mathbb{T}}
\newcommand{\RT}{\ensuremath{RT_k(\T)}}
\newcommand{\R}{\mathbb{R}}
\newcommand{\Pik}{\ensuremath{\Pi_k}}
\newcommand{\Cone}{\ensuremath{C_3}}
\newcommand{\Cextraone}{\ensuremath{C_4}}
\newcommand{\Conne}{\ensuremath{C_1}}
\newcommand{\Ctwo}{\ensuremath{C_2}}
\newcommand{\Mone}{\ensuremath{C_5}}
\newcommand{\Mtwo}{\ensuremath{C_6}}
\newcommand{\CF}{\ensuremath{C_F}}
\newcommand{\Cfive}{\ensuremath{C_7}}
\newcommand{\Ceight}{\ensuremath{C_8}}

\DeclareMathOperator{\dist}{dist}
\DeclareMathOperator{\ddiv}{div}
\DeclareMathOperator{\intt}{int}
\newcommand{\ddivo}{\ensuremath{\rm div{=}0}}
\DeclareMathOperator{\mmid}{mid}
\DeclareMathOperator{\Lip}{\mathcal{L}ip}
\usepackage{mathtools}
\DeclarePairedDelimiterX{\inp}[2]{\langle}{\rangle}{#1, #2}
\newcommand{\tauRT}{\ensuremath{\tau_{RT}}}
\begin{document} 
\title{Stability of mixed FEMs for non-selfadjoint 
indefinite  second-order  linear elliptic PDEs}
\author{C.~Carstensen\thanks{Humboldt-Universit\"at zu Berlin, Berlin, Germany \& Distinguished Visiting Professor, 
 Department of Mathematics, 
 Indian Institute of Technology Bombay, Powai,
 Mumbai 400076, India
         (cc@math.hu-berlin.de).} \and
       Neela Nataraj\thanks{Department of Mathematics, Indian Institute of Technology Bombay, Powai, Mumbai-400076, India (neela@math.iitb.ac.in).}
        \and Amiya K. Pani \thanks{Department of Mathematics, Indian Institute of Technology Bombay, Powai, Mumbai-400076, India (akp@math.iitb.ac.in).}}

\maketitle
\begin{abstract}
For a well-posed  non-selfadjoint 
indefinite  second-order  linear elliptic PDE
with general coefficients $\mathbf A, \mathbf b,\gamma$ in $L^\infty$ 
and symmetric and uniformly positive definite coefficient matrix $\mathbf A$,
this paper proves that  mixed finite element problems  are uniquely solvable  and
the discrete solutions are uniformly bounded, whenever the underlying shape-regular triangulation is sufficiently fine. This applies to the Raviart-Thomas (RT) and
Brezzi-Douglas-Marini (BDM) finite element
families  of any order and in any space dimension and leads to the 
best-approximation estimate in
 $H(\ddiv)\times L^2$ as well as in  in $L^2\times L^2$  up to oscillations.  
 This generalises earlier contributions for piecewise  Lipschitz  continuous coefficients to $L^\infty$ 
 coefficients. 
The compactness argument of Schatz and Wang for the 
displacement-oriented problem  does {\em not} apply  immediately to the mixed formulation in $H(\ddiv)\times L^2$. But it allows the uniform approximation  
of some $L^2$ contributions 
and can be  combined with a recent  $L^2$ best-approximation  result from the medius
analysis.   This technique circumvents any regularity assumption and the 
application of a Fortin interpolation operator.  
\end{abstract}

{\small\noindent\textbf{Keywords}
mixed finite element method,  stability, 
 non-selfadjoint indefinite, general 
linear second-order  elliptic PDE, RT and BDM finite elements, 
best-approximation, medius analysis

\medskip

\section{Introduction}
This section introduces the non-selfadjoint 
indefinite  second-order  linear elliptic PDE and its
mixed formulations. A  brief review of earlier results is followed 
by the assertion of the stability and the best-approximation 
results.

\subsection{Non-selfadjoint indefinite second-order linear elliptic PDEs}
The strong formulations for  second-order elliptic problems with coefficients 
${\bf A},$  {\bf b}, $\gamma$ componentwise in  $L^\infty(\Omega)$ 
and $f\in L^2(\Omega)$ read $ \mathcal{L}_j u_j =f $ a.e. in a polyhedral 
bounded Lipschitz domain $\Omega\subset\R^n$ { with homogeneous Dirichlet boundary condition $u_j= 0$ on $\partial \Omega$ for $j=1,2$ 
and any dimension $n\ge 2$.} 
For all  $v\in H^1_0(\Omega)$, the two differential 
operators (referred to as conservative resp. divergence form throughout this paper) 
read
\begin{eqnarray}\label{eqintroeq1}
 \mathcal{L}_1 v :=  -\nabla \cdot (\mathbf A\nabla v+v\, {\mathbf b})+ \gamma \, v
 \quad\text{and}\quad
  \mathcal{L}_2 v :=  -\nabla \cdot (\mathbf A\nabla v) 
  + {\mathbf b}\cdot \nabla v+ \gamma \, v.
\end{eqnarray}
The assumption on ellipticity means that the 
$n\times n$ coefficient matrix 
$\mathbf A(x)$ is symmetric and positive definite with eigenvalues in one universal
compact interval of positive reals  for a.e. $x\in \Omega$.
This makes  $\mathcal{L}_1, \mathcal{L}_2:H^1_0(\Omega)\to H^{-1}(\Omega)$ 
Fredholm  operators of index zero and their weak formulations 
$a(v,w):=
\inp{\mathcal{L}_1v}{w}_{H^{-1}(\Omega)\times H^1_0(\Omega)} 
=\inp{\mathcal{L}_2 w}{v}_{H^{-1}(\Omega)\times H^1_0(\Omega)} $,
for all $v,w\in H^1_0(\Omega)$,
are dual to each other in the duality bracket 
$\inp{\bullet}{\bullet}_{H^{-1}(\Omega)\times H^1_0(\Omega)}$ of
$H^{-1}(\Omega)$, the dual of $H^1_0(\Omega)$.   

Throughout this  paper, zero eigenvalues are excluded and the kernel (of one of these  operators) $\mathcal{L}_j$  is supposed to be  
trivial, so that 
$ \mathcal{L}_1$ and $\mathcal{L}_2$  are bijections. 
It is known from  the theory of bilinear forms in reflexive Banach spaces
\cite{BBF,Braess}
that this implies well-posedness and  the continuous $\inf$-$\sup$ condition
(e.g., when $H^1_0(\Omega)$ is endowed with the 
norm $\|\nabla\bullet\|$)
\begin{equation}\label{eqdefalpha}
0<\alpha:= \inf_{v\in H^1_0(\Omega)\setminus\{0\}}   
\sup_{v\in H^1_0(\Omega)\setminus\{0\}}   \frac{   a(v,w)}{ \| \nabla v\|\,
\| \nabla w\| } .
\end{equation}
The $\inf$-$\sup$ constant  is the same for the original and the dual problem;
$a(v,w)$ could be replaced by $a(w,v)$ with the same $\alpha$. 
The finite element error analysis is enormously simplified  under additional conditions 
on the coefficients that lead to an ellipticity of $a(\bullet,\bullet)$ and allow an application of 
the  Lax-Milgram lemma \cite{BBF,Braess,BrennerScott}. The present situation of 
a general  non-selfadjoint  indefinite  second-order  linear elliptic  PDE 
avoids any of those assumptions and examines  coefficients in $L^\infty$,
which satisfy the following.
 
 {\bf Assumption (A).} There exist two global constants   $0<\underline{\alpha}\le \overline{\alpha}<\infty$ such that   ${\bf A}\in L^\infty(\Omega;\mathbb{R}^{n\times n})$
 satisfies $\underline{\alpha}\leq \lambda_1({\bf A}(x))\leq \cdots \leq \lambda_n({\bf A}(x))\leq \overline{\alpha}$ for the eigenvalues 
$\lambda_1({\bf A}(x))\le\dots\le \lambda_n({\bf A}(x))$ of the SPD 
${\bf A}(x)$  for a.e. $x\in \Omega$.  The functions  $\mathbf  b,
\mathbf  b_1,\mathbf  b_2 \in L^\infty(\Omega;\R^n)$ and 
$\gamma\in L^\infty(\Omega)$ are componentwise
bounded in the bounded polyhedral Lipschitz domain $\Omega\subset\R^n$.

Given the various applications to porous media and ground-water flow with  rough and oscillating coefficients merely bounded in a well-posed  PDE,
this contribution gives an affirmative answer to the fundamental question whether the mixed finite element method   be used
(and then is stable and  provides best-approximation property at least for fine triangulations).

\subsection{Earlier contributions} 
For conforming finite element discretizations and sufficiently small mesh sizes, \cite{SchatzWang96} establishes the existence and uniqueness of conforming  finite element solutions under assumption (A). 
The mixed formulation  for the conservation (resp. divergence) equation
$ \mathcal{L}_1 u =f $ (resp. $\mathcal{L}_2 u =f $) introduces the flux variable
$\sigma = -\mathbf A\nabla u- u\,{\mathbf b}$ 
(resp. $\sigma = -\mathbf A\nabla u$) and 
seeks { the solution} $x=(\sigma,u) \in H$  to 
\begin{equation}\label{eqmainmixedproblem}
b(x,y)=(f,v)_{L^2(\Omega)}
\quad\text{for all }y=(\tau,v) \in H :=H(\ddiv, \Omega) \times L^2(\Omega)
\end{equation}
with
${\mathbf b}_1:= \mathbf A^{-1} {\mathbf b}$, ${\mathbf b}_2:=0$
(resp. ${\mathbf b}_1:= 0$, ${\mathbf b}_2:= {\mathbf b} \cdot \mathbf A^{-1} $) 
and 
$(\sigma,\tau)_{\mathbf A^{-1}}:=(\mathbf A^{-1}\sigma,\tau)_{L^2(\Omega)}$
in
\begin{align}\label{eqmainbilinearformofthispaper}
b(x,y)&=(\sigma,\tau)_{\mathbf A^{-1}} - (u,\ddiv \tau)_{L^2(\Omega)}
+ (v,\ddiv\sigma)_{L^2(\Omega)} \\ \nonumber 
& +(u, \mathbf b_1\cdot\tau)_{L^2(\Omega)}  
 -(v,\mathbf  b_2\cdot\sigma)_{L^2(\Omega)}
+(\gamma\, u,v)_{L^2(\Omega)}.
\end{align}
The equivalence to the boundary value problems associated with the linear differential operators in \eqref{eqintroeq1}  and their  well-posedness   on the continuous 
level can be found in \cite[Sect.~2]{CCADNNAKP15}.  
 This implies  the  continuous $\inf$-$\sup$ conditions \cite{BBF,Braess}
\begin{equation}\label{eqdefbeta}
0<\beta:= \inf_{x\in H \setminus\{0\}}   \sup_{y\in H\setminus\{0\}}   
\frac{   b(x,y)}{ \| x\|_{H} \| y\|_H  }
= \inf_{y\in H \setminus\{0\}}   \sup_{x\in H\setminus\{0\}}   
\frac{   b(x,y)}{ \| x\|_{H} \| y\|_H  } .
\end{equation}
The  existence and  uniqueness of discrete solutions and optimal $L^2$ error estimates were introduced in \cite{DR82} for sufficiently fine triangulations in two and three
space dimensions  under high regularity  assumptions, where
the pair $(\sigma_h,u_h)$ is approximated in $\RT\times P_k(\T)$ 
with  the Raviart-Thomas (RT) for $2D$ (resp. Raviart-Thomas-Nedelec for $3D$) finite elements.
Global $L^{\infty}$ and global $L^2$ and negative norm estimates for the 
conservation form were discussed in \cite{DR85,GN88}  for smooth coefficients. 

Provided the coefficients $\mathbf A$ and 
${\mathbf{b}}$ are Lipschitz continuous,  
$\gamma$ is piecewise Lipschitz continuous {{and  $H^2$ regularity of the adjoint system}},  an interesting  convergence 
phenomenon for the 
BDM finite element family is clarified in the fairly general 
framework of \cite{demlow02}.

Let $M_k(\T)$ be any RT or BDM finite element space
of degree $k\in\mathbb{N}_0$ and define the  discrete space 
$V(\T):=M_k(\T)\times P_k(\T)\subset H$ based on a shape-regular triangulation
$\T$ with mesh-sizes $\le \delta$, written $\T\in\TT(\delta)$. 

In case  $\mathbf A$ and $\mathbf b$ are globally 
Lipschitz  continuous and $\gamma$ is piecewise Lipschitz continuous,
the convergence results in \cite{demlow02} also establish stability in the sense
\begin{equation}\label{eqmainstabilityresultofthispaper}
0<\beta_0\le \inf_{ \T\in\TT(\delta)} \inf_{x_h\in V(\T)\setminus\{0\}}\sup_{y_h\in 
V(\T)\setminus\{0\}} \frac{ b(x_h,y_h) }{ \| x_h\|_H\|y_h\|_H   } (=: \beta_h) 
\end{equation}
for some positive $\delta$ and $\beta_0$.  
With extra work and refined arguments along the lines of \cite{demlow02}, 
but  with reduced elliptic regularity and 
solution $u_j\in  H^1_0(\Omega)\cap H^{1+s}(\Omega)$ to $ \mathcal{L}_j u_j =f $
for some $s>0$. Those arguments are {\em not} valid under  Assumption~(A).

Modern trends in the mathematics of mixed finite element schemes include local stable projections 
with commuting properties \cite{egsv-2019,eg-2004,eg-2016}; those techniques do not seem to allow the proof of discrete stability 
and best-approximation under assumption (A).

{ Piecewise Lipschitz continuous coefficients with  regularity in $H^{1+s}$ (for some positive $ s$) lead in 
\cite{CCADNNAKP15} to stability for  the lowest-order RT FEM. The equivalence to nonconforming Crouzex-Raviart finite elements
holds more generally \cite{arbogast-chen} and the combination with the arguments from \cite{SchatzWang96}  and \cite{CCADNNAKP15}
might lead to  stability results under the assumption (A) for more examples. In comparison,    the methodology of this paper provides stability for any degree 
$k$ and any dimension $n$ (RT and BDM merely serve as popular model examples).}
 
\subsection{Contribution of this paper}\label{subsecContributionofthispaper}
Under the Assumption~(A) and for 
any RT or BDM  finite element space
$V(\T):=M_k(\T)\times P_k(\T)\subset H:= H(\ddiv,\Omega)\times L^2(\Omega)$
of degree $k\in\mathbb{N}_0$ \cite{BBF,Braess,BrennerScott}, the discrete stability
\eqref{eqmainstabilityresultofthispaper}
is established for small mesh-sizes, where
either  ${\mathbf b}_1:= \mathbf A^{-1} {\mathbf b}$ and ${\mathbf b}_2:=0$
or  ${\mathbf b}_1:= 0$ and ${\mathbf b}_2:= {\mathbf b} \cdot \mathbf A^{-1} $ in 
\eqref{eqmainbilinearformofthispaper}.
 
\begin{theorem}[discrete stability]\label{thm1mainresult}
For each (positive) $\beta_0<\beta$ with $\beta$ from \eqref{eqdefbeta},  there exists $\delta>0$
such that \eqref{eqmainstabilityresultofthispaper} holds. 
\end{theorem}

This theorem implies   \cite{BBF,Braess}
that the mixed finite element problems for the  RT and 
the BDM finite element families  of any degree $k$ and in any space dimension  $n$ 
are 
(i) uniquely solvable, 
(ii) uniformly bounded in $H$, and 
(iii)  fullfil   quasi-optimal 
error estimates in the norm of $H$, 
whenever the underlying shape-regular triangulation 
is sufficiently fine. 

Theorem~\ref{thm1mainresult} and the tools of this paper
lead to  $L^2$ best-approximation  up to oscillations.

\begin{theorem}[$L^2$ best approximation]\label{thm1mainresult2}
Suppose $\delta>0$ satisfies  \eqref{eqmainstabilityresultofthispaper} with 
$b(\bullet,\bullet) $  defined for  general 
$\mathbf  b_1,\mathbf b_2\in L^{\infty}(\Omega;\R^n)$ under 
 Assumption~(A). Assume  $\T\in \TT(\delta)$ and that $x:=(\sigma,u)\in H$ 
(resp.  $x_h\equiv(\sigma_h,u_h)\in V_h:=V(\T):= M_k(\T)\times P_k(\T)$) 
satisfy  
$b(x-x_h,y_h)=0$ for all $y_h\in V_h$. Then the following results (a) and (b) hold. 
\\
(a) There exists a positive constant $\Conne$, which exclusively depends on 
$\beta_0>0$,  the $L^\infty$ norms of (all the components of)  $\mathbf A^{1/2}\mathbf b_1$, 
$\mathbf A^{1/2}\mathbf b_2$, and $\gamma$,  as well as on the shape-regularity of $\mathbb{T}$,
such that  the piecewise mesh size $h_\T$  in $\T$  and  the $L^2$ projection $\Pi_k$ onto $P_k(\T)$ satisfy 
\[
\Conne^{-1}\left( \| \sigma-\sigma_h\|_{\mathbf A^{-1}}+\| u - u_h \| \right)
\le \min_{\tau_h \in M_k(\T)} 
 \|  \sigma -\tau_h \|_{\mathbf A^{-1}} +\| u-\Pik u  \|
+\| h_\T(  1-\Pik)\ddiv\sigma\| .
\]
(b) Suppose $\mathbf b_1=0$ and that the  scalar $\gamma(x)$ is  
Lipschitz continuous in $x\in \intt(T)$, the interior of $T\in\T$, with 
a Lipschitz constant smaller than or equal to  $\Lip(\gamma)$.
Then there exists a  positive constant $\Ctwo$, which depends exclusively depends on 
$\beta_0>0$,   $\| \mathbf A^{1/2}\mathbf b_2\|_{L^\infty(\Omega)}$, 
 $\Lip(\gamma)$, and  the shape-regularity of $\mathbb{T}$,
 such that 
\[
\Ctwo^{-1} \| \sigma-\sigma_h \|_{\mathbf A^{-1}} 
\le \min_{\tau_h\in M_k(\T) } \|  \sigma -\tau_h \|_{\mathbf A^{-1}} 
+ \| h_\T(u-\Pik u)\|+\| h_\T(  1-\Pik)  \ddiv\sigma\|.
\]
\end{theorem}

The additional oscillations   $\| h_\T(u-\Pik u)\|$ and 
$\| h_\T(  1-\Pik)(  \ddiv\sigma)\|$  can be higher-order contributions
and then these terms explain the improved convergence of
one variant for the BDM finite element family in \cite{demlow02}
under Assumption~(A).

\medskip
This article, thus, generalises earlier contributions 
 \cite{CCADNNAKP15}, \cite{demlow02}-\cite{DR85},  \cite{eg-2004}-\cite{GN88} for smooth or piecewise  Lipschitz  continuous coefficients to $L^\infty$  
 coefficients without any further assumptions. 
The compactness argument of Schatz and Wang \cite{SchatzWang96} for the 
displacement-oriented problem  does {\em not} apply  immediately to the mixed formulation in $H(\ddiv)\times L^2$. 
Remark~$12$ below explains that no uniform $L^2$  approximation of the divergence component holds. This paper therefore
compensates the lack of compactness by the computation and analysis of an optimal test function (the dual solution $y$ in \eqref{dual-y}
of Subsection 1.4 below). 
Recent best-approximation for the flux in $L^2$ 
from the medius analysis   \cite{CCGDMS16,HuangXu12}  combines with the compactness  for the (dual) PDE. This and 
a careful shift of the discrete divergence circumvents the aforementioned lack of compactness in the divergence variable. 
In fact, this new methodology avoids any regularity argument and any Fortin interpolation at all.

\subsection{Motivation}\label{subsectionMotivation2}
This subsection outlines the proof of the  discrete  $\inf$-$\sup$ stability 
\eqref{eqmainstabilityresultofthispaper} in an abstract framework to guide 
the reader through the arguments.  
Suppose   $X_h\times Y_h$ is a finite dimensional subspace of  $H\times H$ with dual 
$X_h^* \times Y_h^*$ and  
let $x_h \in S(X_h)$, i.e., $x_h$ belongs to $X_h$ and has  norm $\| x_h\|_H=1$.
Recall  \eqref{eqdefbeta} and the well-posedness of the 
problem \eqref{eqmainmixedproblem}. Then, the dual problem is well-posed as well and 
$\inp{x_h}{\bullet}_H=b(\bullet,y)$ has a unique dual solution 
$y $ in the Hilbert space $(H, \inp{\bullet}{\bullet}_H)$. The continuous 
$\inf$-$\sup$ condition \eqref{eqdefbeta} shows
\begin{equation}\label{dual-y}
\beta\, \| y\|_H\le \| b(\bullet,y)\|_{H^*}=\|x_h\|_H=1,\quad\text{whence }
\|  y\|_H\le 1/\beta
\end{equation}
is bounded. Suppose that  $y_h\in Y_h$ is a close approximation to $y$ with 
$\| y-y_h\|_H\le \varepsilon $ for some positive  $\varepsilon < 1/\| b\|$, where $\|b\|$ is the 
operator norm  of the bilinear form $b(\bullet,\bullet)$. 
Since 
\[
1= b(x_h,y)=b(x_h,y_h)+ b(x_h,y-y_h)\le 
\| b(x_h,\bullet)\|_{Y_h^*} \| y_h\|_H +\varepsilon \| b\|,
\]
it remains to bound $\|y_h\|_H$, e.g., with the triangle inequality
\[
\|y_h\|_H\le \|y\|_H+\|y-y_h\|_H\le 1/\beta+\varepsilon. 
\]
The combination of the previous two displayed formulas gives a lower bound
for $\| b(x_h,\bullet)\|_{Y_h^*}$. Under the 
assumption that $\varepsilon$ is independent of 
$y$ and so of $x_h$, this estimate reads
\begin{equation}\label{eqdefbetah}
\beta \frac{ 1-   \varepsilon \| b\|}{1+\varepsilon \beta}
\le \beta_h:= \inf_{x_h\in X_h \setminus\{0\}}   \sup_{y_h\in Y_h\setminus\{0\}}   
\frac{   b(x_h,y_h)}{ \| x_h\|_{H} \| y_h\|_H  }.
\end{equation}
This proves $\beta_0\le \beta ( 1-   \varepsilon \| b\|)/(1+\varepsilon \beta)$ provided 
the approximation error $\| y-y_h\|_H$ is small independently of $X_h\times Y_h$
and $x_h\in S(X_h)$.  A detailed investigation  in Subsection~3.3 below reveals  that the above strong form of 
a  uniform approximation appears 
{\em neither} available in the norm of  $H=H(\ddiv,\Omega)\times L^2(\Omega)$ 
(cf. Remark~\ref{remarknoapproximationinHddiv}) 
{\em nor} necessary for the stability under assumption {(A)}. { Recent results from a medius analysis  \cite{CCGDMS16,HuangXu12}   and a careful shift of the discrete divergence variable successfully circumvent 
a uniform approximation in $H$.}

\subsection{Structure of the paper}
Section~2 starts with the pre-compactness  for 
uniform approximation and the precise assumptions on the set of 
admissible triangulations $\TT$. The other two preliminary subsections concern the $L^2$ best-approximation of the fluxes and some discrete approximation result for the 
RT finite element family. 

The stability analysis in Section~3 is based on the dual solution $y$ in the conservative formulation characterised in Subsection~3.1. One contribution of $y$ involves the
PDE  $\mathcal{L}_2\phi=g$ and allows for some pre-compacness and uniform approximation in Subsection~3.2. The proof of Theorem~\ref{thm1mainresult}
concludes Section~3.
A combination of the stability result   \eqref{eqmainstabilityresultofthispaper}  with the approximation arguments leads in Section~4 to Theorem~\ref{thm1mainresult2}, which generalises \cite{CCGDMS16,HuangXu12}
to  non-selfadjoint  indefinite  second-order  linear elliptic problems. 

\section{Preliminaries}\label{secpreliminaries}
This section introduces notations used in the paper, fixes the  assumptions on the admissible triangulation $\TT$,  discusses an abstract version of  compactness argument in \cite{SchatzWang96}, 
and then recalls some $L^2$ best-approximation property and  
concludes with an observation for the RT finite element family. 

\subsection{Notation}
Standard notation on Lebesgue  and Sobolev spaces $L^2(\Omega)$, 
$L^\infty(\Omega)$,  $H_0^1(\Omega)$, $H^{-1}(\Omega)\equiv H_0^1(\Omega)^*$,
and  $H(\ddiv,\Omega)$ apply  throughout this paper. 
The $L^2$ scalar product  $(\bullet,\bullet)_{L^2(\Omega) }$
induces the norm $\|\bullet\|:= \|\bullet\|_{L^2(\Omega) }$ and the orthogonality relation
 $\perp$.

Whereas  $\|\bullet\| $ denotes the norm  in $L^2(\Omega)$
with the exception of the abbreviation $\|b\|$ for the
bound of the bilinear form $b(\bullet,\bullet)$, 
the vector space $L^2(\Omega;\R^n)$ is endowed with the weighted 
scalar product $(\bullet,\bullet)_{\mathbf A^{-1}}:=  ({\mathbf A^{-1}}\bullet,\bullet)_{L^2(\Omega)}$ and induced 
norm 
$\|\bullet\|_{\mathbf A^{-1}}:= \|\mathbf A^{-1/2}\bullet\| $ and so, for any $\tau\in L^2(\Omega;\R^n)$, 
is its distance 
$\dist(\tau,M_h):=\min_{\tau_h\in M_h}\| \tau-\tau_h\|_{\mathbf A^{-1}}$
to any subspace $M_h$ of $L^2(\Omega;\R^n)$. 
The  norm $\| (\tau,v)\|_H$ in the 
Hilbert space  $H(\ddiv,\Omega)$ is weighted with
$\mathbf A^{-1}$ in  $L^2(\Omega;\R^n)$ for the flux variable so 
the Hilbert space $H\equiv H(\ddiv,\Omega)\times L^2(\Omega)$ has the weighted scalar product  
$\inp\bullet\bullet_H$ with the induced  norm $\| (\tau,v)\|_H$,  
\begin{equation}\label{eqnorminHofthispaper}
\| (\tau,v)\|_H^2:= \| \tau\|_{\mathbf A^{-1}}^2 +\|\ddiv\tau\|^2+\| v\|^2
\quad\text{for all }(\tau,v)\in H.
\end{equation}
Duality brackets have the dual pairing as an index as in 
$\inp{\bullet}{\bullet}_{H^{-1}(\Omega)\times H^1_0(\Omega)}$ 
above.
To abbreviate the definition of $\inf$-$\sup$ constants throughout this paper, let
$S(V):=\{ v\in V: \|v\|_V=1\}$ for any normed linear space $(V, \|\bullet \|_V)$.

All emerging generic positive constants  $\Conne, \dots, \Ceight$ 
in this paper exclusively depend  on 
$\underline{\alpha}, \overline{\alpha}, \|\mathbf b\|_{L^\infty},
\|\mathbf  b_1\|_{L^\infty},\|\mathbf b_2\|_{L^\infty}, $ and $\|\gamma\|_{L^\infty}$
as well as on  $\alpha$ in 
\eqref{eqdefalpha} and $\beta$ in \eqref{eqdefbeta}  and on 
the class of admissible triangulations $\TT$ specified in Subsection~\ref{subsecpreliminaries1}.

\subsection{Assumptions on the discretization}\label{subsecpreliminaries1}
The finite element spaces are based on admissible triangulations, the set $\TT$
of all of those has certainly infinite  cardinality; the point is that 
the  constants in standard  interpolation error estimates become universal through uniform 
shape regularity.

\begin{definition}[admissible triangulations]
The set of  admissible triangulations  $\TT $ is a set of shape-regular triangulations of the polyhedral bounded Lipschitz domain $\Omega\subset\R^n$  into simplices with uniform shape regularity and arbitrary small mesh sizes. Let 
$h_{\max}(\T):= \max h_\T $ for the piecewise constant mesh-size $h_\T$ for  $\T\in\TT$,
defined by 
$h_\T |_T:=\text{\rm diam}(T)$ in $T\in\T$,
and abbreviate $\T(\delta):=\{\T\in\TT:  h_{\max}(\T)\le \delta\}$.
\end{definition}

Given $\T\in\TT$, let  $P_k(T)$ denote the polynomials of total degree at most $k\in\mathbb{N}_0$ seen as functions 
on $T\in\T\in\TT$ and set  $P_k(\T):=\{ v_k\in L^\infty(\Omega): \forall T\in\T,\; v_k|_T\in P_k(T)\}$.
Let $\Pi_k:L^2(\Omega)\to L^2(\Omega)$ be the $L^2$ projection onto $P_k(\T)$
with respect to $\T\in\TT$.

\begin{definition}[discrete spaces]\label{discrete-space}
Any $\T\in\TT$ is associated to the  finite-di\-men\-sion\-al subspace 
$V(\T)=M_k(\T)\times P_k(\T)$ of  $V:= L^2(\Omega;\R^n)\times L^2(\Omega)$ 
with   $M_k(\T):=\RT$ or $M_k(\T):=BDM_k(\T)$ of order $k\in\mathbb{N}_0$  from \cite{BBF}.
\end{definition}

%
The best-approximation error reads $\dist(v, V(\T)):= \inf\{   \|  v-v_h\| : v_h \in V(\T)\}$
with the weighted $L^2$ norm,  $\| v\|^2=\|\tau\|^2_{\mathbf A^{-1}} +\| w\|^2$ for $v=(\tau,w)\in V$.
The density of smooth functions and standard approximation results for smooth functions proves
the well-known pointwise convergence in the sense that  each  $v\in V$ satisfies \cite{BBF}
\begin{equation}\label{eqccnewrevisionc1}
\lim_{\delta\to 0^+} \sup_{\T\in\TT(\delta)} \dist(v, V(\T))=0.
\end{equation}

\subsection{Pre-compactness}\label{subsecpreliminaries2} 
This subsection adopts the key argument of \cite{SchatzWang96}.

\begin{lemma}[uniform approximation on compact sets]\label{lemmauniformapproximation-CC1}
Suppose that  $K$ is a non-empty pre-compact   subset of $V:= L^2(\Omega;\R^n)\times L^2(\Omega)$  
with \eqref{eqccnewrevisionc1} for each $v\in K$. Then
\begin{equation}\label{eqccnewrevisionc2}
\lim_{\delta\to 0^+} \sup_{v\in K}\sup_{\T\in\TT(\delta)} \dist(v, V(\T))=0.
\end{equation}
\end{lemma}

\begin{proof}
Given any $\varepsilon>0$ and $v\in K$, let $B(v,\varepsilon/2)$ be  the open 
ball in $V$ with center $v$ and radius $  \varepsilon/2$. The open cover 
$\{B(v,\varepsilon/2):v\in K\}$ of the compact set $\overline{K}$ contains a finite sub-cover
and so there exist $k_1,\dots,k_J\in K$ with
$K\subset \bigcup_{j=1,\dots,J} B(k_j,\varepsilon/2)$. 
For each $k_j\in K$, \eqref{eqccnewrevisionc1} leads to $\delta_j>0$ 
such that $\T\in\T(\delta_j)$ implies 
$\dist(k_j,V(\T))<\varepsilon/2 $. Then  $\delta:=\min\{ \delta_1, \dots,\delta_J\} $ implies
$\TT(\delta)\subset \bigcap_{j=1,\dots,J}\TT( \delta_j)$. 
Given any $\T\in\TT(\delta)$ and any 
 $v\in K\subset \bigcup_{j=1,\dots,J} B(k_j,\varepsilon/2)$, there exists  $j\in\{1,\dots,J\}$ 
 with $\| v-k_j \|<\varepsilon/2$.  Since $\T\in  \TT( \delta_j)$, 
 $\dist(k_j,V(\T))<\varepsilon/2 $. This and a triangle
inequality show
\(
 \dist(v,V(\T))\le  \| v-k_j \|+\dist(k_j,V(\T))<\varepsilon/2+\varepsilon/2  =\varepsilon.
\)
\end{proof}

The application of the previous lemma to the finite element approximation of the solution of
the PDE reads as follows.

\begin{lemma}[uniform approximation of solutions]\label{lemmauniformapproximation-CC2}
For any $\varepsilon>0$ there exists some $\delta>0$ such that, given  any $g\in L^2(\Omega)$ and  
the weak solution $\phi\in H^1_0(\Omega) $ to
$\mathcal{L}_2\phi  =  g $  (with $\mathcal{L}_2 $ 
from \eqref{eqintroeq1}), the vector $v:= (\mathbf  A \nabla\phi, \mathbf b\cdot \nabla\phi+\gamma\phi)\in V$
satisfies   
$
\sup_{\T\in\TT(\delta)} \dist( v , V(\T)) 
\le  \epsilon \,\|g\|.
$
\end{lemma}

\begin{proof}
The linear and bounded bijective differential operator 
$\mathcal{L}_2:H^1_0(\Omega)\to H^{-1}(\Omega)$  has a bounded inverse.  The 
embedding $\iota: L^2(\Omega)\hookrightarrow H^{-1}(\Omega)$ is compact and so is
the composition  $\mathcal{L}_2^{-1} \circ \iota   :L^2(\Omega)\to H^1_0(\Omega)$. 
Define the operator  
$T: L^2(\Omega)\to L^2(\Omega;\R^n)\times L^2(\Omega)$  for any $g\in L^2(\Omega)$ by  
\[
T(g):=( \mathbf A \nabla\phi , \mathbf b\cdot \nabla\phi+\gamma\phi)
\quad\text{with}\quad\phi:=\mathcal{L}_2^{-1} g.
\]
Since $\mathcal{L}_2^{-1}\circ\iota$ is compact,  $K:= T(S(L^2(\Omega))$ is
pre-compact in $ V= L^2(\Omega;\R^n)\times L^2(\Omega)$. 
Given  any $\varepsilon>0$  the approximation result  
\eqref{eqccnewrevisionc1} and  Lemma~\ref{lemmauniformapproximation-CC1} 
lead to a positive  
$\delta$ with \eqref{eqccnewrevisionc2}. Consequently,  the assertion 
$\sup_{\T\in\TT(\delta)} \dist( v , V(\T)) \le  \epsilon \,\|g\|$ holds for all 
$g\in S(L^2(\Omega))$ and corresponding 
$v:= (\mathbf  A \nabla\phi, \mathbf b\cdot \nabla\phi+\gamma\phi)\in V$. 
A rescaling proves the result  for all  $g\in L^2(\Omega)$.
\end{proof}

\subsection{$L^2$ best-approximation of the fluxes}\label{subsecpreliminaries3}
The medius analysis of mixed finite element methods employs
arguments from {\it a~priori} and {\it a~posteriori} error analysis \cite{CCGDMS16,HuangXu12}
to prove  new $L^2$ best-approximation results. Recall that  $\Pik$ is the $L^2$ projection onto $P_k(\T)$ and $h_\T$ is the mesh-size  associated to $\T$. 
 
\begin{lemma}[flux $L^2$ best-approximation]\label{fluxL^2best-approximation}
There exists a constant $\Cone$, which depends on the shape-regularity in $\T$, on $\Omega$ and 
on $\underline{\alpha}, \overline{\alpha}$,   such for any $\mathbf p\in H(\ddiv,\Omega)$ 
and any $\T\in\TT$, there exists $\mathbf p_h\in M_k(\T)$ such that 
$\ddiv \mathbf p_h =\Pi_k \ddiv \mathbf p$ and 
\[
\Cone^{-1} \| \mathbf p-\mathbf p_h\|_{\mathbf A^{-1}} 
\le\dist( \mathbf p,M_k(\T))+
\| h_\T(1-\Pi_k)\ddiv \mathbf p\|.
\]
\end{lemma}

This is the $L^2$ best-approximation result from \cite[Lemma 5.1]{HuangXu12}
for mixed finite element approximations  for the unit matrix $\mathbf A$.  Although with
a different focus, the  paper \cite{CCGDMS16} introduces a general framework with
a mesh-dependent norm 
$\|\bullet\|_h$ in  $P_k(\T)$; while  \cite[Eq (3.6)]{egsv-2019} presents a localized refinement of this lemma. 

\begin{proof}
Given $\mathbf p\in H(\ddiv,\Omega)$, 
the right-hand sides $F(w):=(w,\ddiv \mathbf p)_{L^2(\Omega) }$ and 
$G(\mathbf q):= (\mathbf p,\mathbf q) $
lead in  the elliptic  mixed formulation (for the Laplacian) 
\[
(\sigma, \mathbf q) - (u, \ddiv \mathbf q)_{L^2(\Omega) }
+ (w ,\ddiv \mathbf p)_{L^2(\Omega) } 
= G(\mathbf q)+F(w) \quad
\text{for all } (\mathbf  q,w)\in H
\]
to the unique solution $(\sigma,u)\equiv(\mathbf p,0)\in H$. 
Its  straight-forward mixed finite element discretisation  substitutes  
$H$ by $V_h:=M_k(\T)\times P_k(\T)$ and leads to 
a unique discrete solution $(\mathbf p_h,v_h)\in V_h$ with
$\ddiv \mathbf p_h=\Pi_k \ddiv \mathbf p$. This and  \cite[Thm 2.2]{CCGDMS16}
lead to the asserted best-approximation result (in terms of (non-weighted) $L^2$ norms) 
\[
\Cextraone^{-1} \| \mathbf p-\mathbf p_h\|
\le\inf_{{ \mathbf q}_h\in M_k(\T)}   \|  \mathbf p-{ \mathbf q}_h\| +
\| h_\T(1-\Pi_k)\ddiv \mathbf p\|.
\]
The constant $\Cextraone$ from \cite{CCGDMS16,HuangXu12} does not depend on the coefficients 
${\bf A},$  {\bf b}, $\gamma$ but depends on the shape-regularity in $\T$ and on $\Omega$.
The equivalence of norms concludes the proof and  leads to the asserted constant $\Cone$, which depends
on $\Cextraone$ and $\underline{\alpha}, \overline{\alpha}$.
\end{proof} 

\subsection{A discrete approximation result for Raviart-Thomas functions}%
\label{subsecpreliminaries4}
In any space-dimension $n$ and degree $k$,  the RT functions 
satisfy a rather particular approximation estimate with  
the  componentwise $L^2$ projection $\Pi_k$  onto $P_k(\T)$.

\begin{lemma}\label{lemmaonRTinnerapprox}
 Any $\tauRT\in \RT$ satisfies 
\(
\| \tauRT- \Pik \tauRT\| \le  \frac{n}{(n+1)(n+k)}  \|  h_\T    \ddiv \tauRT\|.
\)
\end{lemma}

The proof will be postponed to the appendix because of its focus  on the 
RT finite element shape functions.
The statement of the above lemma fails for the BDM finite element family.

\section{Stability analysis}\label{secThedualapproachtowardsstability}

{This section deals with approximation of fluxes and stability result.} The design of a test function in the proof of a discrete $\inf$-$\sup$ condition 
is based on the characterisation and approximation of a dual solution.

\subsection{Dual solution and conservative formulation}
The inner structure of the dual solution  $y$ exploits  the elliptic PDE
and generates some compactness argument in the subsequent subsection.
Recall that the operator $\mathcal{L}_2 :H^1_0(\Omega)\to H^{-1}(\Omega)$ 
from \eqref{eqintroeq1} is bijective. 

\begin{theorem}[dual solution in conservative formulation]%
\label{thmdualsolutioninconservativeformulation}
Suppose $\mathbf  b_1:= \mathbf  A^{-1}  \mathbf b$  
and $\mathbf  b_2\equiv 0$ a.e. in $\Omega$ in \eqref{eqmainbilinearformofthispaper}.
Then $x=(\sigma,u)\in H$ and $y=(\zeta,z)\in H$ satisfy  
$\inp{x}\bullet_{H}=b(\bullet, y)$ in $H$
if and only if 
\[
\zeta  = \sigma-\mathbf  A \nabla\phi \quad\text{and}\quad 
z= \ddiv\sigma-\phi \quad\text{ a.e. in }\Omega
\]
for the weak solution $\phi\in H^1_0(\Omega)$ to  
$ \mathcal{L}_2\phi=g:=   \mathbf b \cdot  \mathbf  A^{-1} \sigma +(\gamma -1)\,\ddiv\sigma-u\in L^2(\Omega)$.
\end{theorem}

The function $\phi$ originates from a known  $L^2$   orthogonal decomposition 
\begin{equation}\label{eqlemmaorthogonaldecomposition-CC2}
L^2(\Omega;\R^n)= \nabla H^1_0(\Omega) \oplus H(\ddivo,\Omega)
\end{equation}
with \( H(\ddivo,\Omega):=\{ \tau\in H(\ddiv,\Omega):\, \ddiv\tau=0\text{ a.e. in }\Omega\} \). 
The decomposition \eqref{eqlemmaorthogonaldecomposition-CC2} 
is also useful in the proof of equivalence of the displacement
formulation with the differential operators in 
\eqref{eqintroeq1} to the mixed formulations with \eqref{eqmainmixedproblem}.

\medskip

{\em Proof of Theorem~\ref{thmdualsolutioninconservativeformulation}.} 
For the general version of the bilinear form $b(\bullet,\bullet)$, 
the equation $\inp{x}\bullet_H=b(\bullet, y)$ is equivalent to 
$u = \mathbf  b_1\cdot\zeta+\gamma \,z-\ddiv \zeta$ 
a.e. in $\Omega$ and 
\begin{align}\label{dual-solution}
(\tau,\mathbf A^{-1} (\zeta-\sigma) -z \mathbf  b_2 )_{L^2(\Omega) }
+ (z-\ddiv\sigma,\ddiv\tau)_{L^2(\Omega) }=0
\quad\text{for all } \tau\in H(\ddiv,\Omega).
\end{align}
The test with $\tau\in H(\ddivo,\Omega)$ proves that
$\mathbf A^{-1} (\sigma-\zeta) + z \mathbf b_2\perp 
H(\ddivo,\Omega)$ and so \eqref{eqlemmaorthogonaldecomposition-CC2} leads
to $\phi\in H^1_0(\Omega)$ with  
\[
\mathbf  A \nabla\phi= \sigma-\zeta + z \mathbf  A  \mathbf  b_2 
\quad\text{ a.e. in }\Omega.
\]
This identity allows the substitution of 
$\mathbf A^{-1} (\zeta-\sigma) -z \mathbf  b_2$ in the above formula  \eqref{dual-solution}
with general $\tau\in H(\ddiv,\Omega)$. Then,  { an integration by parts shows
the resulting identity} 
$(\phi+z-\ddiv\sigma,\ddiv\tau)=0$. The surjectivity of
 $\ddiv:H(\ddiv,\Omega)\to L^2(\Omega)$ proves
\[
\ddiv\sigma=\phi+z \quad\text{ a.e. in }\Omega.
\]
The combination of the three preceding identities leads to the PDE
\[
-\ddiv (\mathbf  A \nabla\phi ) +  \mathbf  b _1\cdot \mathbf  A \nabla\phi  + \gamma\phi
=-\ddiv( z \mathbf  A  \mathbf  b_2)+ (\mathbf  b_1\cdot \mathbf  A   \mathbf b_2)\, z
+ \mathbf b_1\cdot\sigma +(\gamma -1)\, \ddiv\sigma-u
\]
in the sense of distributions. Since $\mathbf b_2=0$, the right-hand side 
$g$ belongs to $L^2$.
This proves  one direction of the assertion;  the direct proof of the converse 
is omitted.\hfill$\Box$

\begin{remark}[no divergence formulation]\label{remark:12}
The  proof shows the extra term $-\ddiv( z \mathbf  A  \mathbf b_2)\in H^{-1}(\Omega)$
in case \eqref{eqmainbilinearformofthispaper} is considered for non-zero
$\mathbf b_2\in L^\infty(\Omega;\R^n)$.
This term does {\em not}  belong to $L^2(\Omega)$  under  Assumption~(A)
and is, therefore, excluded.  
\end{remark} 

\subsection{Approximation of the fluxes}
The subsequent  lemma describes the uniform approximation of the flux variable 
by a  combination of the 
compactness argument and the $L^2$ best-approximation 
of Subsections~\ref{subsecpreliminaries2} and \ref{subsecpreliminaries3}.

\begin{lemma}[flux approximation]\label{fluxapproximation}
Given any $\varepsilon>0$, there exists $\delta>0$ such that the following holds
for all $\T\in\TT(\delta)$ and $g\in L^2(\Omega)$. There exists some $\mathbf p_h\in M_k(\T)$ that approximates  
$\mathbf p :=  \mathbf A\nabla \phi \in H(\ddiv,\Omega)$ for the weak solution 
$\phi\in H^1_0(\Omega)$ to $\mathcal{L}_2\phi= g$ with
\[
\ddiv\mathbf p_h=\Pik \ddiv\mathbf p \quad\text{and}
\quad \| \mathbf p-\mathbf p_h\|_{\mathbf A^{-1}}\le \epsilon \|g\|. 
\]
\end{lemma}

\begin{proof}
Given  any $\varepsilon>0$ and the constant $\Cone$ from 
Lemma~\ref{fluxL^2best-approximation}, Lemma~\ref{lemmauniformapproximation-CC2}
leads to a positive  
$\delta\le \min\{1, 2^{-1} \epsilon/ \Cone \}$ with 
\[
\sup_{\T\in\TT(\delta)}  \dist\left(
(\mathbf  A \nabla\phi, \mathbf b\cdot \nabla\phi+\gamma\phi), V(\T) \right)
\le  2^{-3/2}\epsilon/ \Cone\,\|g\|
\]
(the distance is with respect
to the weighted norm $\|\bullet\|_{\mathbf A^{-1}}$  in $L^2(\Omega;\R^n)$ and
  $\|\bullet\|$ in $L^2(\Omega)$). 
Lemma~\ref{fluxL^2best-approximation} applies to 
$\mathbf p:=\mathbf A\nabla\phi$ with 
$\ddiv \mathbf p=  \mathbf b\cdot \nabla\phi+\gamma\phi-g\in L^2(\Omega)$
and, for any $\T\in\TT(\delta)$, 
leads to some approximation $\mathbf p_h\in M_k(\T)$   with 
$\ddiv \mathbf p_h =\Pi_k \ddiv \mathbf p$ and
\begin{align*}
\Cone^{-1} \| \mathbf p-\mathbf p_h\|_{\mathbf A^{-1}} 
&\le\dist( \mathbf p,M_k(\T))+
\delta \| (1-\Pi_k)(  \mathbf b\cdot \nabla\phi+\gamma\phi-g ) \| \\
&\le \dist( \mathbf p,M_k(\T))+    
\dist(  \mathbf b\cdot \nabla\phi+\gamma\phi, P_k(\T)) +  \delta \, \|g\|\\
&  \le 2^{1/2} \dist(( \mathbf p,\mathbf b\cdot \nabla\phi+\gamma\phi) , V(\T)) +  \delta \, \|g\|
\le \Cone^{-1}  \epsilon \, \|g\|.
\end{align*}
This concludes the proof.
\end{proof} 

\begin{remark}[no uniform approximation in $H(\ddiv)$]\label{remarknoapproximationinHddiv}
Lemma~\ref{fluxapproximation} does {\em not} state a uniform approximation estimate
for the divergence and, in fact, an estimate of the form 
$ \| \ddiv(\mathbf p-\mathbf p_h)\| \le \epsilon \|g\|$  {\em cannot}  hold in general. 
To see this,  adopt the notation of  the proof of Lemma~\ref{fluxapproximation} and 
a reverse triangle inequality for
\[
 \|g-\Pik g \| - \|\ddiv (\mathbf p- \mathbf  p_h)\| 
 \le   \| (1-\Pi_k)( \mathbf b\cdot \nabla\phi+\gamma\phi )\|  \le \epsilon/(2\Cone)\, \|g\|.
\]
The flux $\mathbf p= \mathbf A\nabla \mathcal{L}_2^{-1}g$ depends on 
$g\in S(L^2(\Omega))$ and so does 
the crucial term $ \|\ddiv (\mathbf p- \mathbf  p_h)\|= \|(1-\Pik)\ddiv \mathbf p\|$.
Hence,  
$\sup\{ \|g-\Pik g \|: g\in S(L^2(\Omega))\}=1$ implies 
\[
 1-  \epsilon/(2\Cone) \le 
 \sup\{   \|   (1-\Pi_k) \ddiv (\mathbf A\nabla \mathcal{L}_2^{-1}g)\| : g\in S(L^2(\Omega))   \}.
\]
Therefore,  the approximation error  $ \|\ddiv (\mathbf p- \mathbf  p_h)\|$ will {\em not}
tend to zero uniformly for all  $g\in S(L^2(\Omega))$ as $\epsilon$ and $\delta$ tend to zero. 
\hfill $\Box$
\end{remark}

\begin{example}[RT approximation in $H(\ddiv)$ for particular $g$]\label{exampleRTapproximationinH(div)}
The  stability analysis in Subsection~\ref{subsectionMotivation2} 
concerns  a discrete $x_h:=(\sigma_h,u_h)$ 
with norm $\|(\sigma_h,u_h)\|_H=1$ and leads in 
Theorem~\ref{thmdualsolutioninconservativeformulation} to the particular 
right-hand side
$g:=   \mathbf b \cdot  \mathbf  A^{-1} \sigma_h +(\gamma -1)\,\ddiv\sigma_h-u_h$
with $\|g\|\le \Mone<\infty $ for the essential supremum
$\Mone^2$ of  $|\mathbf  A^{-1/2}\mathbf b|^2 + |\gamma-1|^2+1$ 
in $\Omega$.   This $g$ allows for a uniform approximation 
of $\mathbf p$ by $\mathbf p_h$ in $H(\ddiv,\Omega)$ for the RT finite element family.

For instance, 
in the extreme case of piecewise constant coefficients, 
$g-\Pik g= \mathbf b \cdot  \mathbf  A^{-1}(1-\Pik)\sigma_h$. With 
$\Mtwo :=\|  \mathbf  A^{-1}\mathbf b  \|_{L^\infty(\Omega)} $, 
Lemma~\ref{lemmaonRTinnerapprox} shows 
\(
\|g-\Pik g\|
\le \delta \Mtwo
\).
The combination with  Lemma~\ref{fluxapproximation} lead to $\mathbf p_h$ with
\(
\| \mathbf p-\mathbf p_h\|_{H(\ddiv,\Omega)}\le \epsilon\,\Mone+ \delta\,\Mtwo .
\)
This and  the arguments  of Subsection~\ref{subsectionMotivation2} lead to  the discrete 
stability \eqref{eqmainstabilityresultofthispaper}.  
\end{example}

\subsection{Proof of Theorem~\ref{thm1mainresult}}\label{sectionproofofthm1}
Given any  $0< \varepsilon< \beta/\| b\| $,   choose $\delta>0$ as in  
Lemma~\ref{fluxapproximation}. Suppose $\T\in\TT(\delta)$ and let 
$x_h=(\sigma_h,u_h)\in V_h:=M_k(\T)\times P_k(\T)$
have norm $\| x_h\|_H=1$ and define 
$g:= \mathbf b \cdot  \mathbf  A^{-1} \sigma_h +(\gamma -1)\,\ddiv\sigma_h-u_h$.
 Replace $x$ by $x_h$ in
Theorem~\ref{thmdualsolutioninconservativeformulation}
and let $\phi\in H^1_0(\Omega)$ solve  
$ \mathcal{L}_2\phi=g$  to define
$\zeta  = \sigma_h-\mathbf  A \nabla\phi $ and
$z= \ddiv\sigma_h-\phi $. Then 
 $y=(\zeta,z)\in H$   is the dual solution and solves 
$\inp{x_h}\bullet_{H}=b(\bullet, y)$ in $H$ for 
 the bilinear form  \eqref{eqmainbilinearformofthispaper} 
(with $\mathbf  b_1:= \mathbf  A^{-1}  \mathbf b$  
and $\mathbf  b_2\equiv 0$ a.e. in $\Omega$).
Lemma~\ref{fluxapproximation}  applies to $\mathbf p:= \mathbf A\nabla\phi$ 
and leads to $\mathbf p_h$
with  $\ddiv\mathbf p_h=\Pik \ddiv\mathbf p$  and
$\| \mathbf p-\mathbf p_h\|_{\mathbf A^{-1}}\le \varepsilon\, \Mone$. 
Hence, $\zeta_h:=\sigma_h- \mathbf p_h $
and $z_h:=\Pik z$ define $y_h:=(\zeta_h,z_h)\in V_h$ with
\begin{equation}\label{eqstabilityofyh}
\| y_h\|_H \le \|y\|_H + \|\zeta-\zeta_h\|_{ \mathbf A^{-1}}\le \| y\|_H+ \varepsilon \, \Mone
\end{equation}
as $\|\ddiv\zeta_h\|=\| \Pik \ddiv \zeta\|\le \|\ddiv \zeta\|$, $\| z_h\|=\|\Pik z\|\le \|z\|$
and  $\ddiv\sigma_h\in P_k(\T)$. { Moreover,}
\[
\| \zeta-\zeta_h\|_{\mathbf A^{-1}}^2+ \| z-z_h\|^2\le 
\varepsilon^2 \, \Mone^2+ \| \phi-\Pik \phi\|^2.
\] 
Piecewise Poincar\'e inequalities  (with the  Payne-Weinberger constant $1/\pi$
for convex domains \cite{payneweinberger}) show
$\| \phi-\Pik \phi\|\le \delta/\pi\| \nabla\phi\|$. 
Recall  that   $H^1_0(\Omega)$ is endowed 
with the seminorm $\|\nabla\bullet\|$ and let $\CF$ denote the constant in the
 Friedrichs inequality $\| \bullet\|\le \CF\|\nabla\bullet\|$ in $H^1_0(\Omega)$. Note
that \eqref{eqdefalpha} leads to   
$\alpha\, \| \nabla\phi\|\le \sup\{ a(\psi,\phi) : \psi\in H^1_0(\Omega) , \;
\|\nabla\psi\|=1\}$. Hence
$ a(\psi,\phi)=\inp{\mathcal{L}_2 \phi}{\psi}_{H^{-1}(\Omega)\times H^1_0(\Omega)} =
\inp{g}{\psi}_{H^{-1}(\Omega)\times H^1_0(\Omega)} \le \CF \|g\|$
implies   $\alpha\, \| \nabla\phi\|\le \CF\,\Mone$. The combination with 
Poincar\'e inequalities shows 
$\| \phi-\Pik \phi\|\le \delta \CF\,\Mone/(\alpha\,\pi)$ and so
\[
\| \zeta-\zeta_h\|_{\mathbf A^{-1}}^2+ \| z-z_h\|^2\le 
\varepsilon^2 \, \Mone^2+ \delta^2 \CF^2\,\Mone^2/(\alpha^2\,\pi^2)=:(\varepsilon')^2.
\]
Since $u_h\perp \ddiv(\zeta-\zeta_h)$ and 
$z-z_h\perp \ddiv\sigma_h$ ($\perp$ denotes orthogonality in $L^2(\Omega)$), 
\begin{align*}
b(x_h,y-y_h)&=(\sigma_h+ \mathbf b u_h,\zeta-\zeta_h)_{\mathbf A^{-1}}   
+(\gamma\, u_h,z-z_h)\\
& \le\varepsilon'\left(
 \|\sigma_h +\mathbf b u_h\|^2_{\mathbf A^{-1}} 
 + \|(1-\Pi_k)( \gamma\, u_h)\|^2\right)^{1/2} \\
 & \le\varepsilon'\left( 2 \|\sigma_h \|^2_{\mathbf A^{-1}}
 +2    \| \mathbf A^{-1/2}\mathbf b \|_{L^\infty(\Omega) }^2  \| u_h\|^2
 +  \| \gamma  \|_{L^\infty(\Omega) }^2  \| u_h\|^2\right)^{1/2}
 \le  \varepsilon' \, \Cfive
 \end{align*}
 with the constant 
$ \Cfive^2:= \max\{ 2  , 2  \| \mathbf A^{-1/2}\mathbf b \|_{L^\infty(\Omega) }^2
 + \| \gamma  \|_{L^\infty(\Omega) }^2\}$.
The arguments of Subsection~\ref{subsectionMotivation2} lead to
\begin{equation}\label{eqcclastinproof1}
1=\| x_h\|^2_H=b(x_h,y)=b(x_h,y_h)+b(x_h,y-y_h)
\le \|b(x_h,\bullet )\|_{V_h^*}\| y_h\|_H +  \varepsilon' \, \Cfive.
\end{equation}
Since $\beta\|y\|_H \le \| b(\bullet, y)\|_{H^*}=\|x_h\|_H=1$ implies 
$\|y\|_H \le 1/\beta$,  \eqref{eqstabilityofyh} reads 
$\| y_h\|_H \le 1/\beta + \varepsilon \, \Mone $. 
This and 
\eqref{eqcclastinproof1}
verify 
\[
1- \varepsilon'\, \Cfive \le \|b(x_h,\bullet )\|_{V_h^*}\| y_h\|_H\le \|b(x_h,\bullet )\|_{V_h^*}
(1/\beta + \varepsilon \, \Mone ).
\]
Since $x_h$ was arbitrary in $S(V_h)$ (with $V_h$ endowed with the norm in $H$),
\[
\beta \frac{1- \varepsilon'\, \Cfive }{1+ \varepsilon\,\beta \, \Mone}
\le \beta_h:= \inf_{x_h  \in S(V_h)}  \sup_ {y_h  \in S(V_h)} b(x_h,y_h ).
\]
Relabelling $\epsilon$ and $\delta$ proves the assertion: For any $0< \beta_0<\beta$,
there exists $\delta>0$ with \eqref{eqmainstabilityresultofthispaper}.
This establishes the theorem for the conservative  version 
of the mixed finite element discretisation 
with $\mathbf  b_1:= \mathbf  A^{-1}  \mathbf b$  
and $\mathbf  b_2\equiv 0$ a.e. in $\Omega$. 

To deduce the same $\inf$-$\sup$ constant for the other variant, 
define $b_1(\bullet,\bullet) $  (resp. $b_2(\bullet,\bullet) $)
by \eqref{eqmainstabilityresultofthispaper} for 
$\mathbf  b_1:= \mathbf  A^{-1}  \mathbf b$  
and $\mathbf  b_2\equiv 0$ 
(resp. ${\mathbf b}_1:= 0$ and ${\mathbf b}_2:= {\mathbf b} \cdot \mathbf A^{-1} $)
a.e. in $\Omega$.  The above proof shows  $0<\beta_0\le\beta_h$ and it is elementary 
to see that 
\begin{align*}
 \beta_h&
=\inf_{ (\tau_h,v_h)  \in S(V_h)}  \sup_ {(\sigma_h, u_h) \in S(V_h)} 
b_1((\tau_h, -v_h),(\sigma_h,-u_h)).
\end{align*}
A direct calculation  shows 
$b_1((\tau, -v),(\sigma,-u))=b_2((\sigma,u), (\tau,v))$
for all $(\sigma,u), (\tau,v)\in H$.
This and a duality argument (singular values of a square matrix 
coincide with those of  its transposed) in the last equality show
\[
\beta_h=\inf_{ (\tau_h,v_h)  \in S(V_h)}  \sup_ {(\sigma_h, u_h) \in S(V_h)} 
b_2((\sigma_h,u_h),(\tau_h,v_h))=
\inf_{x_h  \in S(V_h)}  \sup_ {y_h  \in S(V_h)} b_2(x_h,y_h ).
\]
Hence, the divergence formulation has the same 
discrete $\inf$-$\sup$ constant $\beta_h$.
 \hfill  $\Box$
 
 \begin{remark}[$\delta$ dependence]\label{remark:delta}
The size of $\delta$ in \eqref{eqmainstabilityresultofthispaper} is hidden behind a compactness argument of Lemma \ref{fluxapproximation}. Besides the norms and parameters mentioned in Assumption~(A), 
the mapping properties of  $\mathcal{L}_2^{-1}$ are of relevance as well. A review of the proofs of this paper  shows that there is a finite sub-cover of $S(L^2(\Omega))$ with small balls in $H^{-1}(\Omega)$ that leads to a finite
number of (without loss of generality) smooth functions $k_1,\dots, k_J$ as in the proof of Lemma~\ref{lemmauniformapproximation-CC1}. The size of $\delta$ is related to the approximation properties 
of the weak solutions $\Phi_j$ to  $\mathcal{L}_2\Phi_j=k_j$ a.e.  The regularity properties of 
$\Phi_j\in H^1_0(\Omega)$ 
are {\em not}  characterised for Assumption~(A): In fact, it is unknown whether $\Phi_j$ belongs to any 
$H^{1+s}(\Omega)$ for any $s>0$.
Under Assumption~(B) and reduced elliptic regularity, however,  the afore mentioned approximation properties could be quantified more and reveal further information on $\delta$.
\end{remark}

\section{$L^2$ Best-approximation}
The notation of Theorem~\ref{thm1mainresult2} applies throughout this section with 
continuous and discrete solutions $x=(\sigma,u)$ and $x_h=(\sigma_h,u_h)$.

\subsection{Proof of Theorem~\ref{thm1mainresult2}.a}\label{subsect4.1}
Given $\mathbf p:=\sigma\in H(\ddiv,\Omega)$ and 
$\T\in\TT(\delta)$, choose  $\sigma_h^*:=\mathbf p_h\in M_k(\T)$ as in 
Lemma~\ref{fluxL^2best-approximation}, and define
$e_h:=(\sigma_h- \sigma_h^*, u_h-\Pik u)\in V_h$.
Given $\beta_0>0$ in \eqref{eqmainstabilityresultofthispaper} there exists 
some $y_h=(\tau_h,v_h)\in V_h$ with $\| y_h\|_H=1$ and
\[
\beta_0\, \| e_h\|_H\le b(e_h,y_h)=b(( \sigma_h-\sigma_h^*,u_h-\Pik u) , y_h)=
b(( \sigma-\sigma_h^*,u-\Pik u) , y_h).
\] 
Since  $u-\Pik u\perp \ddiv\tau_h$ and $v_h\perp  \ddiv(\sigma-\sigma^*_h)$, 
the last term is equal to
\[
(\sigma-\sigma^*_h ,\tau_h - v_h \,\mathbf A  \mathbf  b_2)_{\mathbf A^{-1}}  
+(u-\Pik u, \mathbf b_1\cdot\tau_h+\gamma\, v_h)_{L^2(\Omega) } \le \Ceight\, 
\| (\sigma-\sigma^*_h, u-\Pik u) \|_L 
\]
in terms of the weighted $L^2$
norm $\|\bullet\|_L\le\|\bullet\|_H$ in $H$ with 
$\| (\tau_h,v_h )\|_L^2:=\| \tau_h\|_{\mathbf A^{-1}}^2+ \| v_h\|^2
\le 
1$ and with $\Ceight^2= 1+\| \mathbf A^{1/2}\mathbf b_1\|_{L^\infty(\Omega)}^2
+\| \mathbf A^{1/2}\mathbf b_2\|_{L^\infty(\Omega)}^2
+ \|\gamma\|_{L^\infty(\Omega)}^2 $.
Consequently,  
$\| e_h\|_H\le \Big(\Ceight/\beta_0\Big)\, \| (\sigma-\sigma^*_h, u-\Pik u) \|_L $.
Lemma~\ref{fluxL^2best-approximation} shows
\begin{equation}\label{eqlemma9again}
\Cone^{-1}\|\sigma-\sigma^*_h\|_{\mathbf A^{-1}}    \le \dist(\sigma,M_k(\T))
+\| h_\T(1-\Pik)\ddiv\sigma\|.
\end{equation}
This and the distance $\dist_L$ (measured in the norm $\|\bullet\|_L$) lead to
\[
\| e_h\|_H
\le \Ceight\max\{1,\Cone\} /\beta_0\, \left( \dist_L((\sigma,u),V_h)
+\| h_\T(1-\Pik)\ddiv\sigma\|\right) .
\]
This,  \eqref{eqlemma9again}, and a triangle inequality  conclude the proof.\hfill $\Box$

\subsection{Proof of Theorem~\ref{thm1mainresult2}.b}
Throughout this subsection, let $\mathbf  b_1\equiv 0$ and 
$\mathbf  b_2:= \mathbf  A^{-1}  \mathbf b$  
a.e. in $\Omega$ in \eqref{eqmainbilinearformofthispaper}
and  let $f\in L^2(\Omega)$ be a fixed right-hand side for the continuous and 
discrete problem $\mathcal{L}_2u=f$. 

Return to the proof of the previous subsection with $e_h$ and follow the first lines until 
\[
\beta_0\, \| e_h\|_H\le b(e_h,y_h)= 
(\sigma-\sigma^*_h ,\tau_h - v_h \,\mathbf A  \mathbf  b_2)_{\mathbf A^{-1}}  
+(u-\Pik u, \gamma\, v_h)_{L^2(\Omega) }.
\]
Abbreviate  $\overline{\gamma}:=\Pi_0\gamma $ and utilize the Lipschitz continuity 
of the coefficients $\gamma$ on each simplex 
\[
\|\Pi_k(\gamma(u-\Pik u)) \| = \|\Pi_k((\gamma-\overline{\gamma})(u-\Pik u)) \|
\le \Lip(\gamma) \| h_\T(u-\Pik u)\|.
\]
This controls the above term 
$(u-\Pik u, \gamma\, v_h)_{L^2(\Omega) }\le \|\Pi_k(\gamma(u-\Pik u)) \| \,\| v_h\|$
and leads to 
\[
\|\sigma_h^*-\sigma_h\|_{\mathbf A^{-1/2}}\le \| e_h\|_L\le \| e_h\|_H
\le   \Ceight/\beta_0\,  \|\sigma-\sigma^*_h\|_{\mathbf A^{-1}} 
+   \Lip(\gamma)/\beta_0\,   \| h_\T(u-\Pik u)\|.
\]
A triangle inequality in $L^2(\Omega;\R^n)$ 
is followed by  \eqref{eqlemma9again} 
in the proof of
\begin{align*}
\|\sigma-\sigma_h\|_{\mathbf A^{-1/2}} & \le  \Cone(1+\Ceight/\beta_0)
\Bigl(  \dist(\sigma,M_h)+\| h_\T(1-\Pik)\ddiv\sigma\|\Bigr)  \\
&\qquad+   \Lip(\gamma)/\beta_0\,   \| h_\T(u-\Pik u)\|. 
\end{align*}
This completes the rest of the proof. $\qquad \Box$
\subsection{Conservative formulation}
Theorem~2.a includes an error estimate for   the conservative formulation with
${\mathbf b}_1:= \mathbf A^{-1} {\mathbf b}$ and ${\mathbf b}_2:=0$ in \eqref{eqmainbilinearformofthispaper} and 
$\sigma = -\mathbf A\nabla u- u\,{\mathbf b}$ with $\ddiv\sigma=f-\gamma u$.
The refined analog of Theorem~2.b is not expected because of an extra term
exemplified in the extreme case of piecewise constant coefficients $\mathbf b_1$ and 
$\gamma$.  The arguments of Subsection~\ref{subsect4.1} lead to 
 \[
\beta_0\, \| e_h\|_H\le 
b(( \sigma-\sigma_h^*,u-\Pik u) , y_h)=
(\sigma-\sigma^*_h ,\tau_h )_{\mathbf A^{-1}}  
+((u-\Pik u) \mathbf b_1, \tau_h -\Pik \tau_h)_{L^2(\Omega) }.  
\]
The last term is not of higher order for the BDM finite element family 
as pointed out in  \cite{demlow02} through numerical evidence. 
For the RT finite element family,  however, Lemma~\ref{lemmaonRTinnerapprox}
shows  $\|\tau_h -\Pik \tau_h\|{ \lesssim}  \| h_\T\, \ddiv \tau_h\|$ and then leads to a higher-order 
contribution  in the asserted inequality of 
Theorem~\ref{thm1mainresult2}.b as the final result. \hfill  $\Box$

{ The arguments could be generalised, but those
result are of limited relevance } as the convergence order is {\em not} 
generally improved in comparison
with Theorem~2.a. An exception is the 
example of  \cite[Sect 3.5]{cps} (with $\mathbf b=0=\gamma$ on the unit ball) when  
Theorem~\ref{thm1mainresult2}.b guarantees  $O(\delta^2)$ for the $L^2$ flux error for $k=0$.

\section*{Acknowledgements}
The research of the first author has been supported by the Deutsche Forschungsgemeinschaft in the Priority Program 1748 under the project "foundation and application of generalied mixed FEM towards nonlinear problems in solid mechanics" (CA 151/22-2).  
The finalization of this paper has been supported by   DST SERB MATRICS grant No. MTR/2017/000199 (NN), MATRICS grant No.  MTR/201S/000309 (AKP)  and SPARC project  (id 235) entitled {\it the mathematics and computation of plates}.

\bibliographystyle{amsplain}
\bibliography{abstractpaper.bib}

\bigskip

\section*{Appendix: Proof of Lemma~\ref{lemmaonRTinnerapprox}}
For any simplex $T\subset \R^n$,  let $P_k(T;\R^n)$ be the linear space of 
vector-valued polynomials $q_k$ of degree at most $k$ in any component
and let  $\|\bullet\|$ abbreviate the $L^2$ norm $\|\bullet\|_{L^2(T)}$
on $T$.  The  particular structure of the RT function $\tauRT$ 
leads to some polynomial $g\in P_k(T)$ and 
\[
\tauRT = g(x)\, x + p_{k} \quad\text{for all }x\in T 
\quad\text{ and  some } p_k\in P_k(T; \R^n). 
\]
The argument $x$ (will always belong to $T$) is 
often neglected as in $\tauRT:=\tauRT(x)$  or $p_k=p_k(x)$, while (with a small inconsistency, but the right emphasis)  written out in 
the leading term $g(x)\, x$. The latter polynomial is either identically zero or of 
exact degree $k+1$ in the sense  
that $g$  is  a sum of monomials of exact degree $k$. Adopt a multi-index notation with 
$\alpha=(\alpha_1,\dots,\alpha_n)\in \mathbb{N}_0^n$ and the monomial $x^\alpha:=x_1^{\alpha_1}x_2^{\alpha_2}\cdots x_n^{\alpha_n}$ of degree 
$k=|\alpha|:=\alpha_1+\dots+\alpha_n$  for any $x=(x_1,\dots,x_n)\in T$.
With real coefficients $c_\alpha$ for any $\alpha\in \mathbb{N}_0^n$ of degree $|\alpha|=k$, 
\[
g(x)=\sum_{|\alpha|=k}  c_\alpha x^\alpha\quad\text{for all }x\in T.
\]
(The symbol $|\alpha|=k$ under the sum sign abbreviates the set of all multi-indices 
$\alpha=(\alpha_1,\dots,\alpha_n)\in \mathbb{N}_0^n$ of degree $k$).
The divergence 
\[
\ddiv(g(x)\, x) = n\, g(x)+ x\cdot\nabla g(x)
\]
of the vector-valued polynomial $g(x)\, x$ of degree $k+1$
with respect to $x$ is computed with the observation that, for any 
$\alpha\in \mathbb{N}_0^n$ with $|\alpha|=k$,
\[
x\cdot \nabla (x^\alpha ) = \sum_{j=1}^n  x_j \partial (x^\alpha)/\partial x_j=
\sum_{j=1}^n  \alpha_j x^\alpha = k \, x^\alpha.
\] 
Consequently, $x\cdot\nabla g(x)= k\, g(x)$ and 
\[
\ddiv(g(x)\, x) = (n+k)\, g(x)\quad\text{for all }x\in T.
\]
This proves $\ddiv\tauRT=\ddiv(g(x)\, x)+q_{k-1}= (n+k)\, g(x)+q_{k-1}$ 
for some $q_{k-1}\in P_{k-1}(T)$. The comparison with 
$\tauRT = g(x)\, x + p_{k} $ leads to some polynomial remainder 
$r_{k}\in P_{k}(T; \R^n)$ in
\[
\tauRT =  (n+k)^{-1}\, (\ddiv\tauRT)\, x  + r_{k} \quad\text{for all }x\in T . 
\]
In other words, since $\ddiv \tauRT\in P_k(T)$,
\[
(n+k)(1-\Pik) \tauRT=(1-\Pik)\left((\ddiv\tauRT)\, x\right)
=(1-\Pik)\left((\ddiv\tauRT)\, (x-c)\right)
\]
for any constant vector $c$. For instance, 
the center of inertia $c=\mmid(T)$  of $T$ with diameter $h_T$
satisfies  $|x-\mmid(T)|\le \left(n/(n+1)\right) h_T$ for all $x\in T$. 
This leads to 
\begin{align*}
(n+k)\| &\tauRT-\Pik\tauRT\| = \|(1-\Pik) \left( (\ddiv\tauRT) (x-\mmid(T))\right)\| \\
&\le \|(\ddiv\tauRT) (x-\mmid(T))\| \le \left(n/(n+1)\right) h_T\, \|\ddiv\tauRT\|.
\end{align*}
This  proves
$\| \tauRT- \Pik \tauRT\|_{L^2(T)} 
\le  \frac{n \, \ h_T}{(n+1)(n+k)}  \| \ddiv \tauRT\|_{L^2(T)}$.
 \hfill $\Box$
\end{document}